\documentclass{article}
\usepackage[utf8]{inputenc}
\usepackage{amsmath, amsthm, amssymb, fdsymbol,url}
\usepackage{graphicx}
\usepackage{caption}
\usepackage{fullpage}
\usepackage{cite}
\usepackage[noend]{algorithmic}
\theoremstyle{definition}
\newtheorem{definition}{Definition}
\theoremstyle{plain}
\newtheorem{theorem}{Theorem}

\newtheorem{observation}{Observation}
\newtheorem{example}{Example}

\newtheorem{corollary}{Corollary}

\newcommand{\nl}{\textrm{Null}}
\newcommand{\cl}{\textrm{Col}}

\newcommand{\rank}{\textrm{rank}}

\usepackage{tikz}
\usetikzlibrary{shapes,decorations,calc,fit,positioning}
\tikzset{
    node style/.style={circle, draw, minimum size=8mm, inner sep=0pt},
    blacknode/.style={node style, fill=black!70, text=white},
    whitenode/.style={node style, fill=white, text=black},
}
\usepackage{todonotes}

\usepackage{xcolor}

\begin{document}
\title{Counting Cholesky factorizations of the zero matrix over $\mathbb{F}_2$}
\author{Hays Whitlatch and Joshua Cooper}
\maketitle
\begin{abstract}

A square, upper-triangular matrix $U$ is a Cholesky root of a matrix $M$ provided $U^*U=M$, where $\cdot^*$ represents the conjugate transpose when working over the complex field and $U^*=U^T$ over the reals and finite fields.  In this paper, we investigate the number of such factorizations over the finite field with two elements, $\mathbb{F}_2$, and prove the equinumerosity, for each fixed rank, of the Cholesky roots of and the upper-triangular square roots of the zero matrix.  We then provide asymptotics for this count and finish with a few directions for future inquiry.
\end{abstract}

\section{Introduction}

Cholesky factorization is a fundamental tool in linear algebra.  A Hermitian, positive semidefinite matrix $M$  can be decomposed into the product of a lower triangular matrix $L$ and its conjugate transpose $L^*$.  This factorization is known for its numerical stability, allows for efficiently solving linear systems, and is applied in Monte Carlo simulations.

Over finite fields, Cholesky factorizations are unlike the classical setting.  Positivity looks very different (see, for example, \cite{CHW24,guillot2025positivity,vishwakarma2025choleskydecompositionsymmetricmatrices}),  factorizations are generally not unique,  and the existence of a factorization can encode combinatorial properties.  In particular, over $\mathbb{F}_2$, Cholesky factorizations have a surprising connection to graph theory and bioinformatics: the success of a graph theoretical operation called \emph{pressing} is equivalent to the existence of a Cholesky factorization (\cite{cooper2016successful}).  Pressing sequences are close relatives of linear extensions of posets, and their enumeration -- approximately and exactly -- has been a source of interesting questions in the literature (for example, \cite{bixby2015proving, cooper2019uniquely}).  In turn, pressing sequences model sorting by reversal of DNA sequences (\cite{hannenhalli1999transforming, pevznerbook}).  As such, understanding their multiplicities is both a natural algebraic question and one with consequences in other fields.  

In this work, we focus on the simplest but most instructive case: the square zero matrix over $\mathbb{F}_2$.  While the zero matrix admits a unique Cholesky factorization over the real and complex fields, over $\mathbb{F}_2$ and other finite fields, it admits many distinct factorizations. In this paper, we give exact counts, by dimension and rank, of the Cholesky factorizations of the zero matrix over $\mathbb{F}_2$, we 
prove the equinumerosity of Cholesky ``roots'' of the zero matrix with a given rank $r$ and upper-triangular square roots of the zero matrix of rank $r$, 
and we analyze these counts via a combinatorial sum appearing in \cite{ekhad1996number} and provide asymptotics for its growth.  Our main results thus can be interpreted as a new combinatorial interpretation of the sequence \cite[A008964]{oeis}.

To motivate our results, we first explore the graph theoretic operation of pressing.

\begin{definition}
A \textbf{bicolored graph} $G = (G,c)$ is a simple graph $G$ with $c: V(G) \rightarrow \{\text{blue}, \text{white}\}$ which assigns a color to each vertex.\footnote{Some authors use black and white instead.}  The complement of blue is white and the complement of white is blue. For 
$v \in V(G)$,  \textbf{pressing} a blue vertex $v$ is the operation of transforming $(G,c)$ to $(G',c')$, a new bicolored graph in which $G[N_G(v) \cup \{v\}]$ is complemented.  That is, $V(G) = V(G')$ and $$E(G')= E(G) \triangle \begin{pmatrix}
N_G(v) \cup \{v\} \\
2
\end{pmatrix}$$
where $\triangle$ denotes symmetric difference, $N_G(v)$ is the neighborhood of $v$ in $G$, $c'(w)$ is the complement of $c(w)$ for $w \in N_G(v) \cup \{v\}$, and $c'(w) = c(w)$ otherwise.  This definition is illustrated in Example \ref{pressingexample} where one can see that pressing a blue vertex changes its color to white and flips the colors of its neighbors, isolates the pressed vertex, and complements the set of edges induced by its neighbors.
\end{definition}

A sequence of presses is referred to as a \textbf{pressing sequence}.  Since a pressed vertex becomes isolated, it may not be pressed again, and will not be affected by future presses.  Thus, every pressing sequence is finite in length.  If the end result of a pressing sequence is the empty, edgeless, colorless graph then we say that it was a \textbf{successful pressing sequence}.  As it turns out, a pressing sequence is successful exactly when the graph's adjacency matrix (with row and column order dictated by the pressing sequence) has a Cholesky factorization over $\mathbb{F}_2$.  In fact, the Cholesky factorization gives the ``instructions'' for pressing, revealing at each press which vertices would be affected, as we now explain.

\begin{example}\label{pressingexample}
Below is a bicolored graph with successful pressing sequence 1,2,3 (blue vertices are shown in black). 

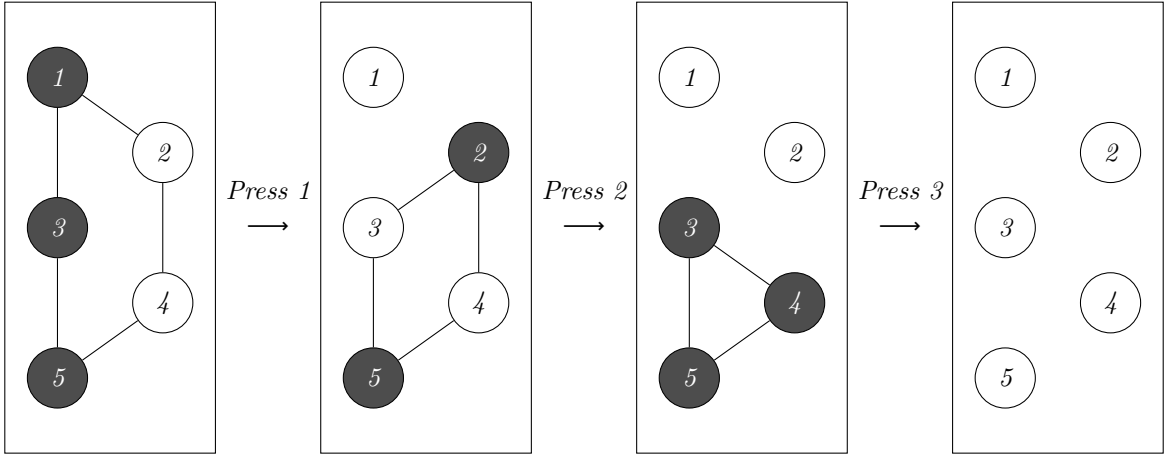
\begin{center}

\begin{tikzpicture}[x=1.4cm, y=1cm] 

\node[blacknode] (a1) at (0,4) {1};
\node[whitenode] (a2) at (1,3) {2};
\node[blacknode] (a3) at (0,2) {3};
\node[whitenode] (a4) at (1,1) {4};
\node[blacknode] (a5) at (0,0) {5};

\draw (a1) -- (a2);
\draw (a1) -- (a3);
\draw (a2) -- (a4);
\draw (a3) -- (a5);
\draw (a4) -- (a5);

\draw (-.5,-1) -- (-.5,5) -- (1.5,5) -- (1.5,-1) -- cycle;

\node at (2,2.5) {\textrm{Press 1}};
\node at (2,2) {\textrm{$\longrightarrow$}};

\node[whitenode] (b1) at (3,4) {1};
\node[blacknode] (b2) at (4,3) {2};
\node[whitenode] (b3) at (3,2) {3};
\node[whitenode] (b4) at (4,1) {4};
\node[blacknode] (b5) at (3,0) {5};

\draw (b2) -- (b3);
\draw (b2) -- (b4);
\draw (b3) -- (b5);
\draw (b4) -- (b5);

\draw (2.5,-1) -- (2.5,5) -- (4.5,5) -- (4.5,-1) -- cycle;

\node at (5,2.5) {\textrm{Press 2}};
\node at (5,2) {\textrm{$\longrightarrow$}};

\node[whitenode] (c1) at (6,4) {1};
\node[whitenode] (c2) at (7,3) {2};
\node[blacknode] (c3) at (6,2) {3};
\node[blacknode] (c4) at (7,1) {4};
\node[blacknode] (c5) at (6,0) {5};

\draw (c3) -- (c4);
\draw (c3) -- (c5);
\draw (c4) -- (c5);

\draw (5.5,-1) -- (5.5,5) -- (7.5,5) -- (7.5,-1) -- cycle;

\node at (8,2.5) {\textrm{Press 3}};
\node at (8,2) {\textrm{$\longrightarrow$}};

\node[whitenode] (d1) at (9,4) {1};
\node[whitenode] (d2) at (10,3) {2};
\node[whitenode] (d3) at (9,2) {3};
\node[whitenode] (d4) at (10,1) {4};
\node[whitenode] (d5) at (9,0) {5};

\draw (8.5,-1) -- (8.5,5) -- (10.5,5) -- (10.5,-1) -- cycle;

\end{tikzpicture}
\end{center}

The adjacency matrix of the graph in Example \ref{pressingexample} (with 1's on the diagonal when the vertex is blue) and the Cholesky factorization corresponding to this pressing sequence (with rows and columns labelled by 1, 2, 3, 4, 5) is 
$$M=\left[\begin{array}{ccccc}
     1 & 1 & 1 & 0 & 0 \\
     1 & 0 & 0 & 1 & 0 \\
     1 & 0 & 1 & 0 & 1 \\
     0 & 1 & 0 & 0 & 1 \\
     0 & 0 & 1 & 1 & 1 \\
\end{array}\right]=\left[\begin{array}{ccccc}
     1 & 1 & 1 & 0 & 0 \\
     0 & 1 & 1 & 1 & 0 \\
     0 & 0 & 1 & 1 & 1 \\
     0 & 0 & 0 & 0 & 0 \\
     0 & 0 & 0 & 0 & 0 \\
\end{array}\right]^T\left[\begin{array}{ccccc}
     1 & 1 & 1 & 0 & 0 \\
     0 & 1 & 1 & 1 & 0 \\
     0 & 0 & 1 & 1 & 1 \\
     0 & 0 & 0 & 0 & 0 \\
     0 & 0 & 0 & 0 & 0 \\
\end{array}\right]=U^TU$$

\end{example}
Observe that in the upper-diagonal matrix $U$, the $j^\textrm{th}$ entry of the $i^\textrm{th}$ row is $1$ exactly when pressing vertex $i$ changed the state of vertex $j$, and in particular is $0$ if $i > \rank(U)$.  For this reason we call $U$ the \emph{instructional} Cholesky factorization of $M$.  This instructional  Cholesky factorization of $M$ is unique and, in \cite{cooper2020new}, the number of instructional  Cholesky factorizations of $\{PMP^T\mid P\textrm{ a permutation matrix}\}$ is discussed.  However, over $\mathbb{F}_2$ a Cholesky factorization of a matrix is not unique in general (only when it is of full rank).  In fact, the matrix in our previous example offers a second Cholesky factorization: $$M=\left[\begin{array}{ccccc}
     1 & 1 & 1 & 0 & 0 \\
     1 & 0 & 0 & 1 & 0 \\
     1 & 0 & 1 & 0 & 1 \\
     0 & 1 & 0 & 0 & 1 \\
     0 & 0 & 1 & 1 & 1 \\
\end{array}\right]=\left[\begin{array}{ccccc}
     1 & 1 & 1 & 0 & 0 \\
     0 & 1 & 1 & 1 & 0 \\
     0 & 0 & 1 & 1 & 1 \\
     0 & 0 & 0 & 0 & 1 \\
     0 & 0 & 0 & 0 & 1 \\
\end{array}\right]^T\left[\begin{array}{ccccc}
     1 & 1 & 1 & 0 & 0 \\
     0 & 1 & 1 & 1 & 0 \\
     0 & 0 & 1 & 1 & 1 \\
     0 & 0 & 0 & 0 & 1 \\
     0 & 0 & 0 & 0 & 1 \\
\end{array}\right]=V^TV$$
We obtain $V$ from $U$ by replacing the $2\times 2$ trailing submatrix $\left[\begin{array}{cc}
     0 & 0 \\
     0 & 0 \\
\end{array}\right]$ with $\left[\begin{array}{cc}
     0 & 1 \\
     0 & 1 \\
\end{array}\right]$.  This works since the latter is also a Cholesky root of the $2\times 2$ zero matrix -- and this construction generalizes, as shown in the next section with Corollary \ref{cor:factorization}.  Say that a graph with vertex set $[n]$ is ``pressable'' if $(1,2,\ldots,k)$ is a successful pressing sequence for some $k \leq n$.  In general, 
when an adjacency matrix of rank $r$ arises from a pressable graph, the number of Cholesky roots of a symmetric matrix over $\mathbb{F}_2$ is equal to the number of Cholesky roots of the all-zeroes square matrix of dimension $n-r$.
There are typically many such roots, because $\approx 58\%$ of symmetric matrices over $\mathbb{F}_2$ are singular (see, for example, \cite{macwilliams69}).

\section{Cholesky Roots over $\mathbb{F}_2$}

In this section we will discuss Cholesky factorizations over $\mathbb{F}_2$.  Our goal is 
to show that the number of distinct Cholesky factorizations of the zero matrix and the number of square roots of the zero matrix are equal.  
Accordingly, in Definition \ref{roots}, we will formally refer to the upper triangular matrix in these factorizations -- typically denoted by $U$ -- as a root (although, for convenience, we will often work with its transpose $L=U^T$) . Before proceeding, we briefly review the behavior of Cholesky roots in the complex setting and contrast this with their behavior over finite fields such as $\mathbb{F}_2$.

Let $n\geq 1$ and let $O^\mathbb{C}_n, I^\mathbb{C}_n$  denote the additive and multiplicative  identity matrices in the space of $n\times n$ complex matrices, respectively.  Suppose $LL^*$ is a Cholesky factorization for $O^\mathbb{C}_n$, then  for all $1\leq i,j\leq n$ the dot product of the $i^\textrm{th}$ row of $L$ and the $j^{\textrm{th}}$ column of $L^*$ must be zero.  However, since the $j^{\textrm{th}}$ column of $L^*$ is simply the complex-conjugate of  the $j^{\textrm{th}}$ row of $L$, we have 
$$
\sum_{k=1}^{n}L[i,k] \overline{L[j,k]} =0 \quad \textrm{ for all }1\leq i,j\leq n .
$$
So by considering $i=j$, we may conclude that $L=O^\mathbb{C}_n$, from which it follows that $O^\mathbb{C}_n$ has a unique Cholesky factorization.  We now return to the finite field with two elements, $\mathbb{F}_2$.

\begin{definition}\label{roots}
Let $M$ be a $n\times n$ symmetric matrix with entries in $\mathbb{F}_2$.  Over this field -- and any finite field -- we say $M$ has a \emph{Cholesky factorization} if there exists a lower-triangular matrix $L$ such that $LL^T=M$ or equivalently if there exists an upper-triangular matrix $U$ such that $U^TU=M$.  In such case, we call $U$ a \emph{Cholesky root of $M$ }.
\end{definition}

For all positive integer $n$, let $O_n$ and $I_n$ denote the additive and multiplicative identity matrices in the space of $n\times n$ matrices over $\mathbb{F}_2$, respectively.  For $r \leq n$, we let $\mathcal{U}_n(r)$ be the set of $n\times n$, rank $r$, upper-triangular matrices with entries from  $\mathbb{F}_2$. For $n\geq 1$ and $r\leq n$ we define

$$\mathcal{A}_n(r) = \{U\in \mathcal{U}_n(r)\mid U^2=I_n \}\quad \textrm{and}\quad \mathcal{A}_n=\bigcup_{0\leq r\leq n} \mathcal{A}_n(r);$$

$$\mathcal{B}_n(r) = \{U\in \mathcal{U}_n(r)\mid U^2=O_n\}  \quad \textrm{and}\quad
\mathcal{B}_n=\bigcup_{0\leq r\leq n} \mathcal{B}_n(r);$$

$$\mathcal{C}_n(r) = \{U\in \mathcal{U}_n(r)\mid U^TU=O_n\}\quad \textrm{and}\quad \mathcal{C}_n=\bigcup_{0\leq r\leq n} \mathcal{C}_n(r)$$

\begin{observation}\label{obs:AandB}
For all $n\geq 1$: $$|\mathcal{A}_n |=|\mathcal{B}_n|$$
\end{observation}
\begin{proof}
Observe that $(X+I_n)^2=X^2+XI_n+I_nX+I^2_n=X^2+I_n$.  Hence, for all $X\in U_n$,  $X^2=O_n$ if and only if $(X+I_n)^2=I_n$.  
\end{proof}

\begin{theorem}\label{thm:BandC}
For all $n\geq 1$ and $r\leq n$: $$|\mathcal{B}_n(r)|=|\mathcal{C}_n(r)| $$
\end{theorem}
\begin{proof}
Observe that
$$
\mathcal{B}_1=\mathcal{B}_1(0)=\left\{\begin{bmatrix}
0
\end{bmatrix}\right\}=\mathcal{C}_1(0)=\mathcal{C}_1.
$$ 
We proceed by induction.  Let $n>1$ and assume that
$\vert\mathcal{B}_{n-1}(r)\vert=\vert\mathcal{C}_{n-1}(r)\vert$ for all $r\leq n-1$. 
Choose and fix a rank $r$, $n\times n$ upper-triangular matrix $B$.  Observe that by Sylvester's rank inequality $\mathcal{B}_n(n)=\mathcal{C}_n(n)=\emptyset$, so we may proceed with the assumption that $r<n$.  Let $B'$ be the $(n-1) \times (n-1)$ principal submatrix of $B$.  Note that
$$
B^2=  \left[ \begin{array}{c|c}
B' & \mathbf{v}  \\ \hline
\mathbf{0}^T & b  \\
\end{array} \right]\left[ \begin{array}{c|c}
B' & \mathbf{v}  \\ \hline
\mathbf{0}^T & b  \\
\end{array} \right]=\left[ \begin{array}{c|c}
B'^2 & B'\mathbf{v} + b\mathbf{v}  \\ \hline
\mathbf{0}^T & b^2  \\
\end{array} \right].$$   Then $B \in \mathcal{B}_{n}$ if and only if $b=0$ and $B'\mathbf{v}=\mathbf{0}$ and
$B'\in \mathcal{B}_{n-1}$.  However $B'\mathbf{v}=\mathbf{0}$ if and only if $\mathbf{v}\in \nl(B')$, the null space of $B'$. If  $B'\in \mathcal{B}_{n-1}$ then the column space of $B'$, $\cl(B')$, must be a subset of $\nl(B')$. It follows that if $B\in \mathcal{B}_n$ then $\mathbf{v}\in \cl(B')$ or $\mathbf{v}\in \nl(B')\setminus \cl(B')$.  Hence, for each $r$:   
\begin{align*}
|\mathcal{B}_{n}(r)| &= |\mathcal{B}_{n-1}(r)|\cdot 2^{r} + |\mathcal{B}_{n-1}(r-1)|\cdot \left(2^{\dim(\nl(B'))}-2^{r-1}\right) \\
& = | \mathcal{B}_{n-1}(r) | \cdot 2^{r} + | \mathcal{B}_{n-1} (r-1)| \cdot\left(2^{n-r}-2^{r-1}\right)
\end{align*}

Choose and fix a rank $r$, $n\times n$ upper-triangular matrix $C$.  Let $C'$ be the $(n-1)\times (n-1)$ principal submatrix of $C$.  Then
$$
C^TC=  \left[ \begin{array}{c|c}
C'^T & \mathbf{0}  \\ \hline
\mathbf{w}^T & c  \\
\end{array} \right]\left[ \begin{array}{c|c}
C' & \mathbf{w}  \\ \hline
\mathbf{0}^T & c  \\
\end{array} \right]=\left[ \begin{array}{c|c}
C'^TC' & C'^T\mathbf{w}   \\ \hline
\mathbf{w}^TC' & \mathbf{w}^T\mathbf{w}+c^2  \\
\end{array} \right].
$$   
Thus, $C \in \mathcal{C}_{n}$ if and only if $\mathbf{w}^T\mathbf{w}+c^2=0$ and  $C'^T \mathbf{w}=\mathbf{0} $ and
$C'\in \mathcal{C}_{n-1}$.  In particular, $c = \mathbf{w}^T \mathbf{w}$ is determined by the choice of $\mathbf{w}$.  Furthermore, $\mathbf{w} \in \nl(C'^T)$, and $\rank(C') \leq \rank(C) \leq \rank(C')+1$ because $C$ is upper triangular.  If $\mathbf{w} \in \cl(C')$, then $\rank(C)=\rank(C')$ because $c = \mathbf{w}^T \mathbf{w} = 0$, since $C'^T C' = 0$ and $\mathbf{w} \in \cl(C') \subseteq \nl(C'^T)$.  On the other hand, if $\mathbf{w} \not \in \cl(C')$, then $\rank(C) > \rank(C')$, so $\rank(C')=\rank(C)-1$.  Thus, it follows that for each $r$:   
\begin{align*}
| \mathcal{C}_{n}(r) | &= |\mathcal{C}_{n-1}(r)| \cdot 2^{r} + | \mathcal{C}_{n-1}(r-1)|\cdot \left(2^{\dim(\nl(C'^T))}-2^{r-1}\right) \\
& = | \mathcal{C}_{n-1}(r) | \cdot 2^{r} + | \mathcal{C}_{n-1}(r-1) | \cdot \left(2^{n-r}-2^{r-1}\right) .
\end{align*}
Therefore, $|B_n(r)|$ and $|C_n(r)|$ satisfy the same recurrence relation with the same initial condition, so are equal.

\end{proof}

In \cite{ekhad1996number} the authors give a count for the number of upper-triangular matrices over $\mathbb{F}_q$ whose square is the zero matrix. Here, we restrict their result to $q=2$.
\begin{theorem}(Theorem 1 in \cite{ekhad1996number})\label{Main Theorem}
\begin{eqnarray}
| \mathcal{B}_{2n} | &=&
\sum\limits_{j}\left[ \binom{2n}{n-3j}-\binom{2n}{n-3j-1}\right]2^{n^2-3j^2-j}\nonumber \\
| \mathcal{B}_{2n+1} | &=&\sum\limits_{j}\left[ \binom{2n+1}{n-3j}-\binom{2n+1}{n-3j-1}\right]2^{n^2+n-3j^2-2j}\nonumber
\end{eqnarray}
\end{theorem}

By Observation \ref{obs:AandB} and Theorem \ref{thm:BandC}, we obtain the following.

\begin{corollary}
For all $n>0$
$$
| \mathcal{A}_n | = | \mathcal{B}_n | = | \mathcal{C}_n | = \sum\limits_{j}\left[ \binom{n}{\lfloor \frac{n}{2}\rfloor-3j}-\binom{n}{\lfloor \frac{n}{2}\rfloor-3j-1}\right]2^{\lfloor \frac{n}{2}\rfloor\lceil \frac{n}{2}\rceil-3j^2-(\lceil \frac{n}{2}\rceil-\lfloor \frac{n}{2}\rfloor+1)j}$$
\end{corollary}

Given a $n\times n$ matrix $M$ with entries in $\mathbb{F}_2$, we let $M_k$ denote the  $k^{\textrm{th}}$ leading principal submatrix of $M$, $1\leq k\leq n$.  We say is $M$ is in leading principal non-singular (LPN) form if $$\det(M_k)=\begin{cases} 1, & \textrm{if } k\leq \textrm{ rank}(M)\\ 0,& \textrm{if  rank}(M)<k\leq n\\ \end{cases}$$
It was shown in \cite{cooper2016successful} that if $M$ is a full-rank, symmetric matrix with entries in $\mathbb{F}_2$ then $M=U^TU$ from some upper-triangular matrix $U$ if and only if $M$ is in LPN form.  Furthermore, this Cholesky factorization is unique.  In \cite{cooper2019uniquely} it was demonstrated that uniqueness fails when $M$ is not full-rank; however, if $M$ is in LPN form then, as discussed in the introduction, the upper-triangular instructional Cholesky root $U$ of the graph whose adjacency matrix is $M$ satisfies $M = U^T U$.  The following corollary then provides a count of all Cholesky roots of $M$.

\begin{corollary}\label{cor:factorization}
Let $A\in {\mathbb{F}_2}^{n\times n}$ of rank $r$ be in leading principal minors form.  The number of distinct Cholesky factorizations for $A$ is $|\mathcal{C}_n(n-r)|$.
\end{corollary}
\begin{proof}
Let $A_{1,1}$ be the principal $r\times r$ submatrix of $A$ and suppose $B^TB=A$ is a Cholesky factorization of $A$.  Then
$$
A= B^TB=\left[ \begin{array}{c|c}
B_{1,1}^T & O  \\ \hline
B_{1,2}^T & B_{2,2}^T  \\
\end{array} \right]\left[ \begin{array}{c|c}
B_{1,1} & B_{1,2}  \\ \hline
O & B_{2,2}  \\
\end{array} \right]=\left[ \begin{array}{c|c}
B_{1,1}^TB_{1,1} & B_{1,1}^TB_{1,2}  \\ \hline
B_{1,2}^TB_{1,1} & B_{1,2}^TB_{1,2}+B_{2,2}^TB_{2,2}  \\
\end{array} \right]
$$
where $B_{1,1}$ is an $r\times r$ matrix.
However, \cite{cooper2019uniquely} demonstrated that $A$ has an (instructional) Cholesky  factorization of the form
$$
A = V^TV=\left[ \begin{array}{c|c}
V_{1,1}^T & O  \\ \hline
V_{1,2}^T & O  \\
\end{array} \right]\left[ \begin{array}{c|c}
V_{1,1} & V_{1,2}  \\ \hline
O & O  \\
\end{array} \right]=\left[ \begin{array}{c|c}
V_{1,1}^TV_{1,1} & V_{1,1}^TV_{1,2}  \\ \hline
V_{1,2}^TV_{1,1} & V_{1,2}^TV_{1,2}  \\
\end{array} \right]
$$
Since $A_{1,1}$ is a full-rank matrix it has a unique Cholesky factorization over $\mathbb{F}_2$ (see proof in \cite{cooper2016successful}).  That is, $B_{1,1}=V_{1,1}$.  Then by invertibility we have $B_{1,2}=\left(B_{1,1}^T\right)^{-1}B_{1,1}^TB_{1,2}=\left(V_{1,1}^T\right)^{-1}B_{1,1}^TB_{1,2}=\left(V_{1,1}^T\right)^{-1}V_{1,1}^TV_{1,2}=V_{1,2}$ and hence 
$$
V_{1,2}^TV_{1,2}=B_{1,2}^TB_{1,2}+B_{2,2}^TB_{2,2}\quad\Rightarrow\quad B_{2,2}^TB_{2,2}=O.
$$
\end{proof}
\section{Asymptotics}
In this section we derive an asymptotic estimate for the quantity described in Theorem \ref{Main Theorem}, in order to better understand how often a matrix has the stated properties as $n \rightarrow \infty$.  As such, we normalize the count by dividing it by $2^{n^2}$, which is  the number of $n\times n$ matrices over $\mathbb{F}_2$.

Recall, from Theorem \ref{Main Theorem}
\begin{eqnarray}
| \mathcal{B}_{2n} | &=&
\sum\limits_{j}\left[ \binom{2n}{n-3j}-\binom{2n}{n-3j-1}\right]2^{n^2-3j^2-j}\nonumber \\
| \mathcal{B}_{2n+1} | &=&\sum\limits_{j}\left[ \binom{2n+1}{n-3j}-\binom{2n+1}{n-3j-1}\right]2^{n^2+n-3j^2-2j}\nonumber
\end{eqnarray}
Consider the case of $2n$.  Then we can write
\begin{align*}
2^{-n^2} | \mathcal{B}_{2n} | &= \sum_{|j| < \sqrt{n}}\left[ \binom{2n}{n-3j}-\binom{2n}{n-3j-1}\right]2^{-3j^2-j} \\
&\qquad +  \sum_{|j| \geq \sqrt{n}}\left[ \binom{2n}{n-3j}-\binom{2n}{n-3j-1}\right]2^{-3j^2-j}.
\end{align*}

Note only $j \in [-n/3,n/3]$ yield non-zero summands, and each summand in the second sum is bounded by  $4^n 2^{-3(j+1/6)^2+1/12} \leq 4^n 2^{-3n+\sqrt{n}}$, so this portion of the total is bounded by $n 2^{-n+\sqrt{n}} = o(1)$.  We write
\begin{align*}
\binom{2n}{n-3j}-\binom{2n}{n-3j-1} &= \frac{(2n)!}{(n-3j)!(n+3j)!} - \frac{(2n)!}{(n-3j-1)!(n+3j+1)!} \\
&= \binom{2n}{n - 3j} \left (1 - \frac{n-3j}{n+3j+1} \right )\\
&= \binom{2n}{n - 3j} \left (\frac{6j+1}{n+3j+1} \right ).
\end{align*}
Next, we use the approximation
\begin{equation} \label{eq1}
    2^{-N} \binom{N}{K} = (1+o(1)) \sqrt{\frac{2}{\pi N}} \exp \left ( -\frac{2 (K-N/2)^2}{N} \right ) 
\end{equation}
which, when $|j| < \sqrt{n}$, yields
\begin{align*}
\binom{2n}{n-3j}-\binom{2n}{n-3j-1} &= \frac{4^n (1+o(1))(6j+1)}{n^{3/2} \sqrt{\pi}} \exp \left ( -\frac{ 9j^2}{n} \right ) 
\end{align*}

Now, considering those $j$ close to $\sqrt{n}$, we have
\begin{align*}
\sum_{\sqrt[4]n < |j| < \sqrt{n}} \left [ \binom{2n}{n-3j}-\binom{2n}{n-3j-1} \right ] 2^{-3j^2-j} &\leq \sum_{\sqrt[4]n < |j| < \sqrt{n}} \frac{4^n (1+o(1))(6j+1)}{n^{3/2} \sqrt{\pi}} e^{-9} 2^{-\Omega(\sqrt{n})} \\
&= O(\sqrt{n}) \frac{4^n O(\sqrt{n})}{n^{3/2}} 2^{-\Omega(\sqrt{n})} \\
&= \frac{4^n}{\Omega(n^{1/2})} 2^{-\Omega(\sqrt{n})} = o(4^n/n^{3/2}).
\end{align*}
Finally,
\begin{align*}
2^{-n^2} | \mathcal{B}_{2n} | &= o \left ( \frac{4^n}{n^{3/2}} \right ) + \sum_{|j| < \sqrt[4]{n}} \frac{4^n (1+o(1))(6j+1)}{n^{3/2} \sqrt{\pi}} \exp \left ( - \frac{9j^2}{n} \right ) 2^{-3j^2-j} \\
&= \frac{4^n}{n^{3/2}\sqrt{\pi}} \left ( \sum_{|j| < \sqrt[4]{n}} (6j+1) 2^{-3j^2-j} + o(1) \right ) = \frac{2^{2n}}{(2n)^{3/2}} (\alpha + o(1)) 
\end{align*}
where $\alpha = 0.28332924469\ldots$.  

Consider the case of $2n+1$.  Then we can write
\begin{align*}
2^{-n^2-n} | \mathcal{B}_{2n+1} | &= \sum_{|j| < \sqrt{n}}\left[ \binom{2n+1}{n-3j}-\binom{2n+1}{n-3j-1}\right]2^{-3j^2-2j} \\
& \qquad +  \sum_{|j| \geq \sqrt{n}}\left[ \binom{2n+1}{n-3j}-\binom{2n+1}{n-3j-1}\right]2^{-3j^2-2j}.
\end{align*}

Note only $j \in [-(n+1)/3,(n+1)/3]$ yield non-zero summands, and each summand in the second sum is bounded by $4^n 2^{-3(j+1/3)^2+1/3} \leq 4^n 2^{-3n+2\sqrt{n}}$, so this portion of the total is bounded by $n 2^{-n+2\sqrt{n}} = o(1)$.  Next, we write
\begin{align*}
\binom{2n+1}{n-3j}-\binom{2n+1}{n-3j-1} &= \frac{(2n+1)!}{(n-3j)!(n+3j+1)!} - \frac{(2n+1)!}{(n-3j-1)!(n+3j+2)!} \\
&= \binom{2n+1}{n - 3j} \left (1 - \frac{n-3j}{n+3j+2} \right )\\
&= \binom{2n+1}{n - 3j} \left (\frac{6j+2}{n+3j+2} \right ).
\end{align*}
Once again applying (\ref{eq1}) yields
\begin{align*}
\binom{2n+1}{n-3j}-\binom{2n+1}{n-3j-1} &= \frac{2 \cdot 4^n (1+o(1))(6j+2)}{n^{3/2} \sqrt{\pi}} \exp \left ( \frac{ -9j^2+3j}{n+1/2} \right ) 
\end{align*}
Now, considering those $j$ close to $\sqrt{n}$, we have 
\begin{align*}
\sum_{\sqrt[4]n < |j| < \sqrt{n}} \left [ \binom{2n+1}{n-3j}-\binom{2n+1}{n-3j-1} \right ] 2^{-3j^2-2j} &\leq \sum_{\sqrt[4]n < |j| < \sqrt{n}} \frac{2 \cdot 4^n (1+o(1))(6j+2)}{n^{3/2} \sqrt{\pi}} e^{-9+o(1)} 2^{-\Omega(\sqrt{n})} \\
&= O(\sqrt{n}) \frac{4^n O(\sqrt{n})}{n^{3/2}} 2^{-\Omega(\sqrt{n})} \\
&= \frac{4^n}{\Omega(n^{1/2})} 2^{-\Omega(\sqrt{n})} = o(4^n/n^{3/2}).
\end{align*}
Finally,
\begin{align*}
2^{-n^2-n} | \mathcal{B}_{2n+1} | &= o \left ( \frac{4^n}{n^{3/2}} \right ) + \sum_{|j| < \sqrt[4]{n}} \frac{2 \cdot 4^n (1+o(1))(6j+2)}{n^{3/2} \sqrt{\pi}} \exp \left ( - \frac{9j^2+3j}{n+1/2} \right ) 2^{-3j^2-2j} \\
&= \frac{2 \sqrt{2} \cdot 2^{2n+1}}{(2n+1)^{3/2}\sqrt{\pi}} \left ( \sum_{|j| < \sqrt[4]{n}} (6j+2) 2^{-3j^2-2j} + o(1) \right ) = \frac{2^{2n+1}}{(2n+1)^{3/2}} (\beta + o(1)) 
\end{align*}
where $\beta = 0.336936271\ldots$.

Thus, we have the following.

\begin{theorem}
For $n \geq 1$,
$$
| \mathcal{B}_n | = (1+o(1)) \frac{2^{n^2/4+n}}{n^{3/2}} \left \{ 
\begin{array}{ll}
\alpha & \textrm{if $n$ even} \\
\beta' & \textrm{if $n$ odd},
\end{array}
\right .
$$
where 
$$
\beta' = \frac{2^{5/4}}{\sqrt{\pi}} \sum_{j \in \mathbb{Z}} (6j+2)2^{-3j^2-2j}
$$
and
$$
\alpha = \frac{2^{3/2}}{\sqrt{\pi}} \sum_{j \in \mathbb{Z}} (6j+1) 2^{-3j^2-j} .
$$
\end{theorem}
Note that $|\mathcal{B}_n|$ is not only a vanishing fraction of the total number of $0-1$ matrices, but is in fact a vanishing proportion of the $2^{n^2(1/2+o(1))}$ upper-triangular $0-1$ matrices.  Strangely, the constants $\alpha, \beta' \approx 0.28$ agree to several decimal places, but are in fact are unequal\footnote{This has been verified rigorously by bounding the tails of these sums.}: $7\cdot 10^{-7} < |\beta'-\alpha| < 8 \cdot 10^{-7}$ .  We have no satisfactory explanation for this near-miss.

\section{Future Work}
We have seen that for the special case that a binary square matrix is in leading principal minors form, then the number of Cholesky factorizations it admits is given by $|\mathcal{C}_k|$  (here $k$ is the nullity of the matrix).  This count is incorrect, however, when the matrix is not in leading principal minors form.  For example $$\left[\begin{array}{cc}0&0\\0&1\\\end{array}\right]^T\left[\begin{array}{cc}0&0\\0&1\\\end{array}\right]=\left[\begin{array}{cc}0&0\\0&1\\\end{array}\right]=\left[\begin{array}{cc}0&1\\0&0\\\end{array}\right]^T\left[\begin{array}{cc}0&1\\0&0\\\end{array}\right]$$ but the number of Cholesky roots of the $1\times 1$ zero matrix (over $\mathbb{F}_2$) is $1$.  Thus, we ask: how many Cholesky factorizations do $n \times n$ matrices have over $\mathbb{F}_2$?  In particular, if the matrix is upper-triangular, can this count be expressed in terms of the quantities $|\mathcal{C}_k|$?

Another interesting question would be to investigate if there is a \emph{natural} mapping between $\mathcal{B}_n$ and $\mathcal{C}_n$.  In this paper we show that $|\mathcal{B}_n(r)|=|\mathcal{C}_n(r)|$ for all $0\leq r\leq n$ but we do not give an explicit bijection between $\mathcal{B}_n(r)$ and $\mathcal{C}_n(r)$.  

Finally, it is worth mentioning that the bijection between the three sets (the Cholesky roots of zero, the upper-triangular roots of zero, and the upper-triangular roots of the identity) may not extend to other finite fields (in part because $(X+I)^2=X^2+I$ is unique to fields of characteristic two).  It follows that to count the number of Cholesky roots of a zero matrix over other finite fields one would likely need different techniques than the ones employed above.
\newpage

\bibliographystyle{abbrv}
\bibliography{bibliography.bib}
\end{document}